\documentclass[11pt]{amsart}	

\usepackage{fullpage}

\usepackage{amssymb, amsmath, amsthm, color, amsfonts, bm}
\usepackage[hidelinks]{hyperref}

\usepackage{paracol, calc, blindtext, xcolor}
\usepackage{mathabx}
\usepackage{mathtools}

\usepackage{mathrsfs}

\newtheorem{theorem}{Theorem}
\numberwithin{theorem}{section}
\numberwithin{equation}{section}

\newtheorem{definition}[theorem]{Definition}

\newtheorem{remark}[theorem]{Remark}

\newcommand{\Z}{\mathbb{Z}}

\renewcommand{\H}{\mathbb{H}}

\allowdisplaybreaks

\begin{document}

\author{Savana Ammons}
\address{Harvey Mudd College}
\email{sammons@g.hmc.edu}

\author{Young Jin Kim}
\address{Reed College}
\email{ykim2113@gmail.com}

\author{Laura Seaberg}
\address{Haverford College}
\email{lseaberg42@gmail.com}

\author{Holly Swisher}
\address{Department of Mathematics, Kidder Hall 368, Oregon State University, Corvallis, OR 97331-4605}
\email{swisherh@math.oregonstate.edu}

\title{An analogue of $k$-marked Durfee symbols for strongly unimodal sequences}

\subjclass[2010]{05A17, 11P81, 11P82, 11P83, 11P84, 11F37}
\keywords{partitions, $k$-marked Durfee symbols, strongly unimodal sequences, rank generating functions, quantum modular forms}

\thanks{This work was partially supported by the National Science Foundation REU Site Grant DMS-1757995, and Oregon State University.}

\begin{abstract} 
In a seminal 2007 paper, Andrews introduced a class of combinatorial objects that generalize partitions called $k$-marked Durfee symbols.  Multivariate rank generating functions for these objects have been shown by many to have interesting modularity properties at certain vectors of roots of unity.  Motivated by recent studies of rank generating functions for strongly unimodal sequences, we apply methods of Andrews to define an analogous class of combinatorial objects called $k$-marked strongly unimodal symbols that generalize strongly unimodal sequences.  We establish a multivariate rank generating function for these objects, which we study combinatorially.  We conclude by discussing potential quantum modularity properties for this rank generating function at certain vectors of roots of unity.
\end{abstract}

\maketitle

\section{Introduction and Statement of Results}

\subsection{Partitions and $k$-marked Durfee symbols}\label{partitions}\label{intro-p}
A \emph{partition} of a positive integer $n$ is any nonincreasing sequence of positive integers called \emph{parts} that sum to $n$; we further define the empty set to be the sole partition of $0$.  Partitions have been a rich source of study from many mathematical and physical perspectives, which is partly due to their powerful connection to the theory of modular forms.  One can see this connection immediately due to the following relationship between the generating function for the partition counting function $p(n)$ and Dedekind's eta function $\eta(\tau)$, 
\begin{equation}\label{eta-p}
\sum_{n\geq 0}p(n)q^n  = q^{\frac{1}{24}}\eta(\tau)^{-1}.
\end{equation}
Here $p(n)$ counts the number of partitions of $n$, and $\eta(\tau) = q^{\frac{1}{24}}\prod_{n=1}^\infty (1-q^n)$ is a weight $1/2$ modular form for $q=e^{2\pi i \tau}$, $\tau\in \H$.  However, the combinatorial rank function for partitions demonstrates perhaps an even more striking relationship between partitions and modularity.  Dyson \cite{Dyson} defined the {\em rank} of a partition to be its largest part minus its number of parts, and conjectured that the rank could be used to combinatorially explain Ramanujan's famous partition congruences modulo 5 and 7.  This was later proved by Atkin and Swinnerton-Dyer \cite{ASD}. 

The \emph{partition rank function} $N(m,n)$ counts the number of partitions of $n$ with rank equal to $m$.  The two variable generating function for $N(m,n)$ may be expressed as the following $q$-hypergeometric series 
\begin{align}\label{rankgenfn}  
\sum_{m\in \Z} \sum_{n\geq 0}N(m,n) z^m q^n =   \sum_{n=0}^\infty \frac{q^{n^2}}{(zq;q)_n(z^{-1}q;q)_n} =: R_1(z;q),
\end{align}   
where $N(m,0)=\delta_{m0}$, in terms of the Kronecker delta function $\delta_{ij}$, and the $q$-Pochhammer symbol $(a;q)_n$ is defined by 
\[
(a;q)_n:=\prod_{j=1}^n (1-aq^{j-1}),
\]
for $n$ either a nonnegative integer or infinity.

Specializing $R_1(z;q)$ at $z=1$ recovers the partition generating function 
\[
R_1(1;q) = \sum_{n\geq 0}p(n)q^n
\]
from which we can observe modularity via \eqref{eta-p}.  Letting $z=-1$ in $R_1(z;q)$ gives
\begin{equation}\label{r1mock2}
R_1(-1;q) =   \sum_{n=0}^\infty \frac{q^{n^2}}{(-q;q)_n^2},
\end{equation}
which is one of Ramanujan's original third order mock theta functions.  Bringmann and Ono \cite{BO} showed that letting $z=\omega$, for any nonidentity root of unity $\omega$, yields that $R_1(\omega;q)$ is a weight $1/2$ mock modular form.  

Partitions can be represented visually using \emph{Ferrers diagrams}, where each part is represented by a horizontal row of dots, which are left-justified and decreasing from top down. The largest square of dots within the Ferrers diagram of a partition is called the \emph{Durfee square} of the partition.  In a 2007 paper, Andrews \cite{Andrews} defined Durfee symbols, which provide an alternate representation of a partition.  For each partition, the Durfee symbol encodes the side length of the Durfee square and in addition the lengths of the columns to the right of as well as the rows beneath the Durfee square.  For example, below we show the $5$ partitions of $4$, followed by the Ferrers diagrams with highlighted Durfee squares, and then the associated Durfee symbols.  

\begin{figure}[h!]
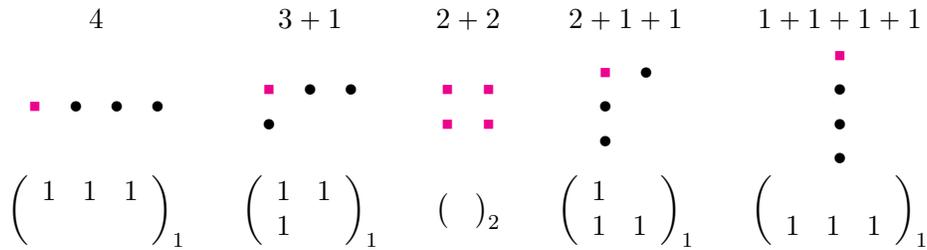
\caption{Partitions, Ferrers Diagrams, and Durfee Symbols for $n=5$}\label{FerrersDurfee} 
\[
\begin{array}{ccccc}
4 & 3+1 & 2+2 & 2+1+1 & 1+1+1+1 \\
\begin{array}{llll} \color{magenta}\sqbullet \color{black} & \bullet & \bullet & \bullet \end{array} & 
\begin{array}{lll} \color{magenta}\sqbullet \color{black}  & \bullet & \bullet \\ \bullet & &  \end{array}  &   
\begin{array}{ll}  \color{magenta}\sqbullet \color{black}  & \color{magenta}\sqbullet \color{black}  \\ \color{magenta}\sqbullet \color{black}  & \color{magenta}\sqbullet \color{black}  \end{array}  & 
\begin{array}{ll} \color{magenta}\sqbullet \color{black}  & \bullet \\ \bullet & \\ \bullet &  \end{array} & \begin{array}{l} \color{magenta}\sqbullet \color{black}  \\ \bullet \\ \bullet \\ \bullet  \end{array} \\
\hspace{2mm} \left( \begin{array}{lll} 1 & 1 & 1 \\  & &  \end{array} \right)_1 \hspace{2mm} & \hspace{2mm} \left( \begin{array}{ll} 1 & 1  \\  1 &  \end{array} \right)_1 \hspace{2mm} & \hspace{2mm} \left( \begin{array}{l}  \\    \end{array} \right)_2 \hspace{2mm} & \hspace{2mm} \left( \begin{array}{ll} 1 & \\  1 & 1 \end{array} \right)_1 \hspace{2mm} & \hspace{2mm} \left( \begin{array}{lll} & & \\ 1 & 1 & 1  \end{array} \right)_1 \hspace{2mm}\\
\end{array}
\]
\end{figure}

We can think of the parts on the top or bottom row as a partition $\alpha$ or $\beta$, respectively. Let $\ell(\alpha)$ denote the number of parts in a partition $\alpha$.  The \emph{rank} of a Durfee symbol is 
\begin{equation}\label{DSrank}
\ell(\alpha) - \ell(\beta),
\end{equation}
the length of the partition in the top row, minus the length of the partition in the bottom row,  which gives Dyson's rank of the corresponding partition.

Andrews \cite{Andrews} modifies Durfee symbols by defining objects called $k$-marked Durfee symbols using $k$ copies of the integers.  He further defines $k$ different rank statistics for the $k$-marked Durfee symbols. Moreover, letting $\mathcal{D}_k(m_1,m_2,\dots, m_k;n)$ denote the number of $k$-marked Durfee symbols of $n$ with $i$th rank equal to $m_i$, Andrews establishes a $k+1$-variable rank generating function $R_k(x_1, \ldots, x_k;q)$ which may be expressed in terms of $q$-hypergeometric series, analogous to (\ref{rankgenfn}).
Complete definitions of these objects are given in \S \ref{comb}. 

Bringmann \cite{Bringmann} showed that the function $R_2(1,1;q)$ is a quasimock modular form. Bringmann, Garvan, and Mahlburg \cite{BGM} expanded on this by showing that $R_k(1,...,1; q)$ is a quasimock modular form for all $k \geq 2$. In 2013, Folsom and Kimport \cite{FK} went on to prove that for more general vectors of roots of unity, $R_k(\omega_1, \ldots, \omega_k; q)$ with $k \geq 2$ is a type of mixed mock modular form. Then in 2018, Folsom, Jang, Kimport, and the fourth author \cite{FJKS, FJKS2} proved that $R_k(\omega_1, \ldots, \omega_k; q)$ with $k \geq 2$ has quantum modularity properties for these more general vectors of roots of unity.

\subsection{Strongly Unimodal Sequences}\label{intro-su}
There is a related combinatorial object which has also exhibited interesting connections to modular forms theory.  A strongly unimodal sequence of size $n$ is a list of positive integers $a_1, \ldots, a_s$ that sum to $n$ such that 
\[
a_1<a_2<\cdots < a_k > a_{k+1} > \cdots >a_s,
\]
where $a_k$ is called the \emph{peak}.   We write $u(n)$ to count the number of strongly unimodal sequences of size $n$.  

As with partitions, strongly unimodal sequences have been a fruitful source of study from a variety of perspectives.  Similar to Ferrers diagrams for partitions, strongly unimodal sequences can be visualized graphically by representing each part $a_i$ as a column of dots, ordering by index from left to right.   For each strongly unimodal sequence we can define a symbol analogous to the Durfee symbol of a partition that encodes the size of the peak of the strongly unimodal sequence, as well as the length of the columns to the right and left of the peak.  We call such symbols \emph{strongly unimodal symbols} in this paper.  For example, below we show the $4$ strongly unimodal sequences of size $4$ with highlighted peaks, followed by their diagrams and associated strongly unimodal symbols.

\begin{figure}[h!]
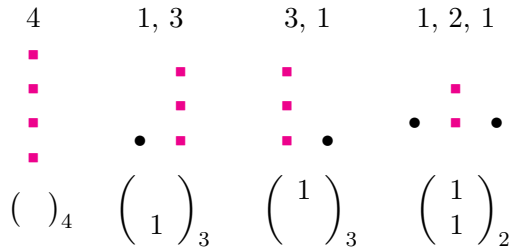
\caption{Strongly Unimodal Sequences, Diagrams, and Symbols for $n=4$}\label{UnimodalStick} 

\[
\begin{tabular}{cccc}
$4$ & $1$, $3$ & $3$, $1$ & $1$, $2$, $1$\\
$\begin{array}{c} {\color{magenta} \sqbullet} \\ {\color{magenta} \sqbullet} \\ {\color{magenta} \sqbullet} \\ {\color{magenta} \sqbullet} \\  \end{array}$ & 
$\begin{array}{cc}  & {\color{magenta} \sqbullet} \\ & {\color{magenta} \sqbullet} \\ \bullet & {\color{magenta} \sqbullet}  \end{array}$ & 
$\begin{array}{cc}   {\color{magenta} \sqbullet}& \\ {\color{magenta} \sqbullet} & \\ {\color{magenta} \sqbullet} & \bullet  \end{array}$ & 
$\begin{array}{ccc}  & {\color{magenta} \sqbullet} & \\ \bullet &  {\color{magenta} \sqbullet} & \bullet  \end{array}$ \\
\hspace{2mm}$\left( \begin{array}{l}  \\  \end{array} \right)_4$ & 
\hspace{2mm}$\left( \begin{array}{l} {}  \\  1   \end{array} \right)_3$ \hspace{2mm}& 
\hspace{2mm}$\left( \begin{array}{l}  1 \\ {}  \end{array} \right)_3$ \hspace{2mm}& 
\hspace{2mm}$\left( \begin{array}{l} 1 \\  1 \end{array} \right)_2$ \\
\end{tabular}
\]
\end{figure}

From a partition theoretic perspective (see \cite{BM} for instance), it is natural to define the \emph{rank} of a strongly unimodal sequence of size $n$ to be the number of terms to the right of the peak minus the number of terms to the left of the peak.  We write $u(m,n)$ to count the number of strongly unimodal sequences of size $n$ with rank $m$.  Similarly to \eqref{rankgenfn}, the generating function for $u(m,n)$ can be expressed as a $q$-hypergeometric series 
\[
\sum_{m\in \mathbb{Z}}\sum_{n\geq 0} u(m,n)z^mq^n = \sum_{n\geq 0} (-zq;q)_n(-zq^{-1};q)_nq^{n+1} =: U(z;q).
\]
Specializing $U(z;q)$ at $z=1$ recovers the strongly unimodal generating function 
\[
U(1;q) = \sum_{n\geq 0}u(n)q^n,
\]
which Andrews \cite{Andrews2} showed could be expressed in terms of two mock theta functions.  In a beautiful paper by Bryson, Ono, Pitman, and Rhoades \cite{BOPR}, they show that $U(-1;q)$ is a mock and quantum modular form which has a duality relationship with Kontsevich's ``strange" function, one of Zagier's original examples of a quantum modular form \cite{Zagier}.  Moreover, $U(\pm i;q)$ is a mock theta function.  There has also been related work connecting rank generating functions for strongly unimodal sequences with mock and quantum modular or Jacobi forms (see \cite{BFR, KLL, BF} for example).

\subsection{Combinatorial work of Andrews and our analogous results}

When Andrews \cite{Andrews} introduced $k$-marked Durfee symbols, he demonstrated a variety of interesting combinatorial properties that they satisfy, including a combinatorial explanation for congruences of the symmetrized $k$th moment functions for partition ranks.  Motivated by the above discussion of combinatorial rank generating functions in the context of modularity, our project was to construct objects analogous to $k$-marked Durfee symbols in the setting of strongly unimodal sequences and study their combinatorial properties with an eye toward connections to modularity.

Andrews \cite{Andrews} established the following combinatorial rank generating function for $k$-marked Durfee symbols. 

\begin{theorem}[Andrews {\cite[Thm. 10]{Andrews}}] \label{Andrews1}
Let $\mathcal{D}_k(m_1,m_2,\dots, m_k;n)$ count the number of $k$-marked Durfee symbols of $n$ with $i$th rank equal to $m_i$.  Then for $k\geq 1$,
\[
\sum_{m_i \in \Z}\sum_{n>0} \mathcal{D}_k(m_1, m_2, ..., m_k; n) x_1^{m_1}x_2^{m_2}...x_k^{m_k}q^n = R_k(x_1, x_2,..., x_k; q),
\]
where $R_1(x;q)$ is defined in \eqref{rankgenfn} and for $k\geq 2$, 
\begin{multline*}
R_k(x_1, \ldots, x_k;q) := \\ \mathop{\sum_{m_1 > 0}}_{m_2,\dots,m_k \geq 0} 
\frac{q^{M_k^2 + (M_1 + \dots + M_{k-1})}}
{(x_1q;q)_{m_1} \!\! \left(\frac{q}{x_1};q\right)_{m_1} \!\!\! (x_2 q^{M_1};q)_{m_2 + 1} \left(\frac{q^{M_1}}{x_2};q\right)_{m_2+1} \!\!\! \!\!\!  \cdots(x_k q^{M_{k-1}};q)_{m_k+1} \!\!  \left(\frac{q^{M_{k-1}}}{x_k};q\right)_{ m_k+1}},
\end{multline*}
where $M_j:= m_1 + m_2 + \cdots + m_j$ for each $1\leq j \leq k$.
\end{theorem}

In \S \ref{comb}, we review the full definition of $k$-marked Durfee symbols, and provide our construction of analogous $k$-marked strongly unimodal symbols.  
Our first result is to establish a rank generating function for $k$-marked strongly unimodal symbols analogous to Theorem \ref{Andrews1}.

\begin{theorem}\label{UnDef}
Let $\mathcal{U}_k(m_1, m_2, ..., m_k; n)$ count the number of $k$-marked strongly unimodal symbols of size $n$ with $i$th rank equal to $m_i$.  Then for $k\geq 1$, 
\[
\sum_{m_i \in \Z}\sum_{n\geq 1} \mathcal{U}_k(m_1, m_2, ..., m_k; n) x_1^{m_1}x_2^{m_2}...x_k^{m_k}q^n = U_k(x_1, x_2,..., x_k; q),
\]
where 
\begin{multline*}
U_k(x_1, x_2,..., x_k; q) := \sum_{m_1,\dots, m_k \geq 1} q^{M_1 +  \cdots + M_{k}}
\cdot (1+x_1^{-1}q^{M_1})(1+x_2^{-1}q^{M_2})\cdots  (1+x_{k-1}^{-1}q^{M_{k-1}}) \\
\cdot (-x_1q;q)_{m_1-1}(-x_1^{-1}q;q)_{m_1-1}(-x_2q^{M_1+1};q)_{m_2-1}(-x_2^{-1}q^{M_1+1};q)_{m_2-1}\\
\cdots(-x_kq^{M_{k-1}+1};q)_{m_k-1}(-x^{-1}_k q^{M_{k-1}+1};q)_{m_k-1},
\end{multline*}
and $M_j:= m_1 + m_2 + \cdots + m_j$ for each $1\leq j \leq k$.
\end{theorem}

We note here that this is not the first generalization of $k$-marked Durfee symbols.  In work of Bringmann, Lovejoy, and Osburn \cite{BLO}, they construct a generalization related to overpartition pairs.  Further, Alfes, Bringmann, and Lovejoy \cite{ABL} have considered this same generalization in the odd setting.  In both cases, automorphic properties of the generating functions were demonstrated.  

Some of the combinatorics Andrews \cite{Andrews} pursued for $k$-marked Durfee symbols was from the perspective of self-conjugation.  The \emph{conjugate} of a partition can be obtained from its Ferrers diagram by simply constructing parts from the columns rather than the rows.  When the resulting partition is unchanged we call it \emph{self-conjugate}.  From this we can observe that self-conjugate partitions must have rank $0$.  As Andrews describes in \cite{Andrews}, Sylvester and Durfee \cite{Sylvester} were the first to study self-conjugate partitions, and it was this which originally led to the idea of Durfee squares as well as the proof that self-conjugate partitions are in bijection with partitions into distinct odd parts.  Andrews \cite{Andrews} defines a \emph{self-conjugate $k$-marked Durfee symbol} to be one in which the top and bottom rows are identical, and proves the following combinatorial result.

\begin{theorem}[Andrews {\cite[Thm. 14]{Andrews}}] \label{Andrews2}
The number of self-conjugate $k$-marked Durfee symbols of $n$ is equal to the number of partitions of $n$ into distinct unmarked odd parts, as well as $k-1$ differently marked $(k-1)$-marked even parts each at most twice the number of odd parts.
\end{theorem}  

We similarly define a \emph{self-conjugate $k$-marked strongly unimodal symbol} to be one in which the top and bottom rows are identical.  Let $\mathcal{SCU}_k(n)$ denote the number of self-conjugate $k$-marked strongly unimodal symbols of $n$, and write the generating function as
\[
SCU_k(q):= \sum_{n\geq 0} \mathcal{SCU}_k(n)q^n.
\]
When $k=1$ (see \eqref{SCUgen_fn2} and preceeding discussion), this generating function has the form
\[
SCU_1(q) = \sum_{n\geq 1}q^n(-q^2;q^2)_{n-1}. 
\]

\begin{remark}\label{mocktheta}
The set of self-conjugate strongly unimodal sequences is in bijection with the set of partitions into odd parts where each odd part of size up to the largest part must occur at least once.  This can be realized by reading across the rows of the diagram for a self-conjugate strongly unimodal sequence to construct an appropriate partition into odd parts.  This explains combinatorially the following interpretation of Ramanujan's $3$rd order mock theta function $\psi(q)$ as stated in Bryson et al \cite{BOPR},
\[
\psi(q) = \sum_{n\geq 1}\frac{q^{n^2}}{(q;q^2)_n} = \sum_{n\geq 1}q^n(-q^2;q^2)_{n-1} = SCU_1(q),
\] 
and shows that $\mathcal{SCU}_1(n)$ can be interpreted combinatorially as the number of partitions of $n$ into odd parts where each odd part of size up to the largest part must occur at least once.
\end{remark}

Our next result provides a more general combinatorial interpretation for $SCU_k(q)$ when $k\geq 2$ analogous to Theorem \ref{Andrews2}.  In order to state it, we make the following definition.

\begin{definition}\label{omegaepsilondef}
Let $\omega_k(n)$ count the number of partitions of $n$ into at least $k$ unmarked odd parts such that every odd part less than the largest part appears at least once, as well as $k-1$ differently marked and distinctly valued $(k-1)$-marked even parts (which may repeat) such that each even part is less than twice the number of odd parts and the total number of even parts is odd.
Similarly, let $\epsilon_k(n)$ count the same as above, except where the total number of even parts is even.
\end{definition}

\begin{theorem}\label{selfconjcounting}
For $k\geq 2$, 
\[
SCU_k(q)=\sum_{n\geq0}(-1)^k(\omega_k(n)-\epsilon_k(n))q^n,
\]
where $\omega_k(n)$ and $\epsilon_k(n)$ are defined in Definition \ref{omegaepsilondef}.
\end{theorem}

The rest of the paper is organized as follows.  In \S \ref{comb}, we review the combinatorial construction of $k$-marked Durfee symbols and define $k$-marked strongly unimodal symbols.  In \S \ref{proofs}, we prove Theorems \ref{UnDef} and \ref{selfconjcounting}.  In \S \ref{conclusion} we conclude with some lingering open questions.

\section{Combinatorial Constructions}\label{comb}

Recall from \S \ref{partitions} that a Durfee symbol for a partition of $n$ encodes the side length of the Durfee square and in addition the lengths of the columns to the right of as well as the rows beneath the Durfee square, as is demonstrated in Figure \ref{FerrersDurfee}.  We now give the full definition of $k$-marked Durfee symbol for a positive integer $k$.

\begin{definition}[Andrews \cite{Andrews}] \label{kMarkedDurfeeDef}
A \emph{$k$-marked Durfee symbol} is a Durfee
symbol using $k$ copies of positive integers, denoted $\{1_1, 2_1, \ldots \}, \ldots, \{1_k, 2_k, \ldots \}$, for the parts in both rows. Additionally, when $k\geq 2$ the following are required.
\begin{enumerate}
\item In each row the sequence of parts and the sequence of subscripts are nonincreasing.
\item In the top row, each of $1, 2, \dots, k-1$ appears as the subscript of some part. 
\item If $M_j$ is the largest part with subscript $1\leq j \leq k-1$ in the top row, then all parts in the bottom row with subscript $1$ lie in the interval $[1, M_1]$, with subscript $2$ lie in $[M_1, M_2],$ ... with subscript $k-1$ lie in $[M_{k-2}, M_{k-1}]$, and with subscript $k$ lie in $[M_{k-1}, M_k]$, where $M_k$ is the side length of the Durfee square of the corresponding partition.
\end{enumerate}
\end{definition}

When we write a $k$-marked Durfee symbol it is convenient to separate the parts with distinct subscripts with vertical lines.  We can think of the collective parts with given subscript $j$ on the top or bottom row as a partition $\alpha^j$ or $\beta^j$, respectively.  In addition, one way to visualize $k$-marked Durfee symbols is through a Ferrers diagram in which parts corresponding to different subscripts have different colors.  We demonstrate this below with an example of a $3$-marked Durfee symbol of $55$.

\begin{figure}[h!]\caption{A $3$-marked Durfee symbol of $55$}\label{3-MDurfee} 
\[
\left(
\begin{array}{cc|ccc|c}
4_3 & 4_3 & 3_2 & 3_2 & 2_2 & 2_1 \\
       & 5_3 &        & 3_2 & 2_2 & 2_1
\end{array}
\right)_5
=:
\left(
\begin{array}{c|c|c}
\alpha^3 & \alpha^2 & \alpha^1 \\
\beta^3 & \beta^2 & \beta^1
\end{array}
\right)_5
\]
\medskip
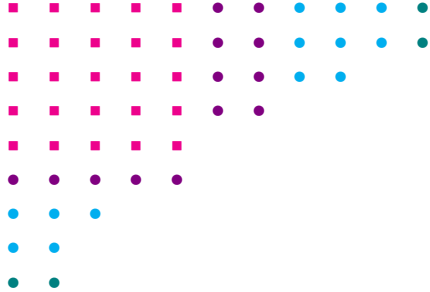
\[
\begin{array}{ccccccccccc}
{\color{magenta} \sqbullet}  & {\color{magenta} \sqbullet} & {\color{magenta} \sqbullet} & {\color{magenta} \sqbullet} & {\color{magenta} \sqbullet} & {\color{violet} \bullet} & {\color{violet} \bullet}& {\color{cyan} \bullet} & {\color{cyan} \bullet} & {\color{cyan} \bullet}   & {\color{teal} \bullet}  \\
{\color{magenta} \sqbullet} & {\color{magenta} \sqbullet} & {\color{magenta} \sqbullet} & {\color{magenta} \sqbullet} & {\color{magenta} \sqbullet} & {\color{violet} \bullet} & {\color{violet} \bullet}& {\color{cyan} \bullet} & {\color{cyan} \bullet} & {\color{cyan} \bullet}   & {\color{teal} \bullet}  \\
{\color{magenta} \sqbullet} & {\color{magenta} \sqbullet} & {\color{magenta} \sqbullet} & {\color{magenta} \sqbullet} & {\color{magenta} \sqbullet} & {\color{violet} \bullet} & {\color{violet} \bullet}& {\color{cyan} \bullet} & {\color{cyan} \bullet} & &  \\
{\color{magenta} \sqbullet} & {\color{magenta} \sqbullet} & {\color{magenta} \sqbullet} & {\color{magenta} \sqbullet} & {\color{magenta} \sqbullet} & {\color{violet} \bullet} & {\color{violet} \bullet}&&& \\
{\color{magenta} \sqbullet} & {\color{magenta} \sqbullet} & {\color{magenta} \sqbullet} & {\color{magenta} \sqbullet} & {\color{magenta} \sqbullet} & &&&& \\
{\color{violet} \bullet} & {\color{violet} \bullet}& {\color{violet} \bullet} & {\color{violet} \bullet} & {\color{violet} \bullet}& & & & & & \\
{\color{cyan} \bullet} & {\color{cyan} \bullet} & {\color{cyan} \bullet} & &&&&&&&\\
{\color{cyan} \bullet} & {\color{cyan} \bullet}& & &&&&&&& \\
{\color{teal} \bullet} & {\color{teal} \bullet} & & &&&&&&& \\
\end{array}
\]
\end{figure}

Using $k$ copies of the integers naturally allows for the definition of $k$ rank statistics on $k$-marked Durfee symbols, generalizing the Durfee symbol rank given in \eqref{DSrank}.

\begin{definition}[Andrews \cite{Andrews}] \label{AndrewsRankDef}
Let $\gamma$ be a $k$-marked Durfee symbol and let $\alpha^j$, $\beta^j$ denote the partitions corresponding to subscript $j$ in the top and bottom rows, respectively.  The \textit{$j$th rank} of $\gamma$, denoted $\rho_j(\gamma)$, is defined by
\[
\rho_j(\gamma)=
\begin{cases}
\ell(\alpha^j) - \ell(\beta^j) -1  \; &  j <k \\
\ell(\alpha^n) - \ell(\beta^n) \; &  j=k.
\end{cases}
\]
\end{definition}

We note here that the extra $1$ is subtracted when $j\neq k$ because in Definition \ref{kMarkedDurfeeDef} it is required that each subscript $1, 2, \dots, k-1$ must appear in the top row.  Moreover, observe that when $k=1$ this recovers Dyson's rank of a partition.  In our example from Figure \ref{3-MDurfee}, we see the $3$rd rank is $1$, the $2$nd rank is $0$, and the $1$st rank is $-1$.  As in \S \ref{partitions}, we let $\mathcal{D}_k(m_1,m_2,\dots, m_k;n)$ denote the number of $k$-marked Durfee symbols of $n$ with $i$th rank equal to $m_i$.

We make the following definition for $k$-marked strongly unimodal symbols, analogous to Definition \ref{kMarkedDurfeeDef}.

\begin{definition}\label{UnimRules}
A $k$-marked strongly unimodal symbol is a strongly unimodal symbol using $k$ copies of positive integers (denoted with a subscript) for parts in both rows.  Additionally, when $k\geq 2$ the following are required.
\begin{enumerate}
\item \label{Rule1} In each row the parts are strictly decreasing and the subscripts are nonincreasing.
\item \label{Rule2} In the top row, each of $1, 2, \dots, k-1$ appears as the subscript of some part. 
\item \label{Rule3} If $M_j$ is the largest part with subscript $1\leq j \leq k-1$ in the top row and $M_0:=0$, then all parts in the bottom row with subscript $1 \leq j \leq k-1$ lie in the interval $[M_{j-1}+1, M_j]$, and those with subscript $k$ lie in $[M_{k-1}+1, M_k-1]$, where $M_k$ is the size of the peak of the corresponding strongly unimodal sequence.  
\end{enumerate}
\end{definition}

The $k$-marked strongly unimodal symbols can be represented analogously to that of $k$-marked Durfee symbols.  We demonstrate this below with an example of a $3$-marked strongly unimodal symbol of $18$.

\begin{figure}[h!]\caption{A $3$-marked strongly unimodal symbol of $18$}\label{3-MUnimodal} 
\[
\left(
\begin{array}{c|cc|c}
4_3 & 3_2 & 2_2 & 1_1 \\
3_3 &        & 2_2 & 1_1
\end{array}
\right)_5
=:
\left(
\begin{array}{c|c|c}
\alpha^3 & \alpha^2 & \alpha^1 \\
\beta^3 & \beta^2 & \beta^1
\end{array}
\right)_5
\]
\medskip
\[
\begin{array}{ccccccccccc}
&&& {\color{magenta} \sqbullet} &&&& \\
&&& {\color{magenta} \sqbullet} & {\color{violet} \bullet} &&& \\
&&{\color{violet} \bullet} & {\color{magenta} \sqbullet} & {\color{violet} \bullet} &{\color{cyan} \bullet}&& \\
&{\color{cyan} \bullet}&{\color{violet} \bullet} & {\color{magenta} \sqbullet} & {\color{violet} \bullet} &{\color{cyan} \bullet}&{\color{cyan} \bullet}& \\
 {\color{teal} \bullet} &{\color{cyan} \bullet}&{\color{violet} \bullet} & {\color{magenta} \sqbullet} & {\color{violet} \bullet} &{\color{cyan} \bullet}&{\color{cyan} \bullet}& {\color{teal} \bullet} \\
\end{array}
\]
\end{figure}

We can define $k$ ranks on $k$-marked strongly unimodal symbols exactly as in the Durfee case.  Since the definition is the same, we use the notation $\rho_j(\gamma)$ in both settings.

\begin{definition}\label{RankDef}
Let $\gamma$ be a $k$-marked strongly unimodal symbol and let $\alpha^j$, $\beta^j$ denote the partitions corresponding to subscript $j$ in the top and bottom rows, respectively.  The \textit{$j$th rank} of $\gamma$, denoted $\rho_j(\gamma)$, is defined by
\[
\rho_j(\gamma)=
\begin{cases}
\ell(\alpha^j) - \ell(\beta^j) -1  \; &  j <k \\
\ell(\alpha^n) - \ell(\beta^n) \; &  j=k.
\end{cases}
\]
\end{definition}

We can observe that when $k=1$, this recovers the rank of a strongly unimodal sequence.  In our example from Figure \ref{3-MUnimodal}, we see the $3$rd rank is $0$, the $2$nd rank is $0$, and the $1$st rank is $-1$.  We let $\mathcal{U}_k(m_1,m_2,\dots, m_k;n)$ denote the number of $k$-marked strongly unimodal symbols of $n$ with $j$th rank equal to $m_j$.

\section{Proofs of Results}\label{proofs}

In this section we prove our two main theorems, beginning with Theorem \ref{UnDef}.

\begin{proof}[Proof of Theorem \ref{UnDef}]
The proof is analogous to that of Andrews \cite[Theorem 10]{Andrews}. Consider an arbitrary $k$-marked strongly unimodal symbol $D$. As in Definition \ref{UnimRules}, let $M_k$ be the size of the peak of the corresponding strongly unimodal sequence, let $M_j$ be the largest part with subscript $1\leq j \leq k-1$ in the top row of $D$, and set $M_0:=0$. We define positive integers $m_1, \ldots, m_k$ associated to $D$ by setting for each $1\leq j \leq k$,
\begin{equation}\label{mdef}
m_j:=M_j-M_{j-1}.
\end{equation}

Let $\alpha^j$, $\beta^j$ denote the partitions corresponding to subscript $j$ in the top and bottom rows, respectively, as in Definition \ref{RankDef}. We next show how to generate the pairs $\alpha^j$, $\beta^j$ beginning with $j=1$.  

To generate $\alpha^1$ and $\beta^1$, we observe that by condition \ref{Rule3} of Definition \ref{UnimRules} the parts of $\alpha^1$ and $\beta^1$ must lie in $[1,M_1]$ and be distinct.  Also, the part $M_1$ must exist in $\alpha^1$.  Furthermore, we need to track the $1$st rank by counting the number of parts other than $M_1$ in $\alpha^1$, and subtracting the number of parts in $\beta^1$.  Thus since $M_1=m_1$, the parts in $\alpha^1$ and $\beta^1$ are generated by
\[
q^{M_1}(-x_1q;q)_{m_1-1}(-x_1^{-1}q;q)_{m_1} = q^{M_1}(1+x_1^{-1}q^{M_1})(-x_1q;q)_{m_1-1}(-x_1^{-1}q;q)_{m_1-1}.
\]

To generate $\alpha^2$ and $\beta^2$, we observe that by condition \ref{Rule3} of Definition \ref{UnimRules} the parts must lie in $[M_1+1, M_2]$ and be distinct.  Also, the part $M_2$ must exist in $\alpha^2$.   Furthermore, we need to track the $2$nd rank.  Thus, the parts in $\alpha^2$ and $\beta^2$ are generated by
\[
q^{M_2}(1+x_2^{-1}q^{M_2})(-x_2q^{M_1+1};q)_{m_2-1}(-x_2^{-1}q^{M_1+1};q)_{m_2-1}.
\]

For general $1\leq j \leq k-1$, we have that the parts are distinct and lie in $[M_{j-1}+1, M_{j}]$, and the part $M_{j}$ must exist in $\alpha^j$.  Thus to track the $j$th rank we see that the parts in $\alpha^j$ and $\beta^j$ are generated by
\[
q^{M_j}(1+x_j^{-1}q^{M_j})(-x_2q^{M_{j-1}+1};q)_{m_j-1}(-x_2^{-1}q^{M_{j-1}+1};q)_{m_j-1}.
\]

The parts in $\alpha^k$ and $\beta^k$ are distinct and lie in $[M_{k-1}+1, M_k-1]$.  Since $\alpha^k$ is allowed to be empty, the parts in $\alpha^k$ and $\beta^k$ are generated more simply by
\[
(-x_kq^{M_{k-1}+1};q)_{m_k-1}(-x^{-1}_k q^{M_{k-1}+1};q)_{m_k-1}.
\]
Finally, the peak is generated by $q^{M_k}$.

As all of these factors generate their respective parts of $D$, their product will generate the entirety of $D$. To account for all possible values for each $m_i$ and the size of $U$, we sum over all $m_i$'s with $m_i\geq 1$ for $1\leq i \leq k$ as well as over all $n \geq 1$. The result follows.
\end{proof}

We now prove Theorem \ref{selfconjcounting}.

\begin{proof}[Proof of Theorem \ref{selfconjcounting}]

This proof follows the method of Andrews \cite[Theorem 14]{Andrews}. We first derive the generating function for $\mathcal{SCU}_k(n)$ using a similar argument to that in the proof of Theorem \ref{UnDef}.   

Consider an arbitrary self-conjugate $k$-marked strongly unimodal symbol $D$, letting $\alpha^j$ denote the partition corresponding to subscript $j$ that occurs in both the top and bottom rows of $D$.  Recall the definition of $M_k$, for $0\leq j \leq k$ given in Definition \ref{UnimRules}, and define the positive integers $m_j$ for $1\leq j \leq k$ by $m_j:=M_j-M_{j-1}$.

The parts in $\alpha^1$ are distinct, lie in $[1,M_1]$, and the part $M_1$ must occur.  Thus the parts in $\alpha^1$ are generated by $q^{M_1}(-q;q)_{m_1-1}$.  To generate both copies of $\alpha^1$ (those from both the top and bottom row of $D$), we need two copies of each part; thus
\[
q^{2M_1}(-q^2;q^2)_{m_1-1}
\] 
generates both copies of $\alpha^1$.  For general $1\leq j \leq k-1$, the two copies of $\alpha^j$ are generated by
\[
q^{2M_j}(-q^{2(M_{j-1}+1)}; q^2)_{m_j-1},
\]
while the two copies of $\alpha^k$ which are not required to have a part are generated by
\[
(-q^{2(M_{k-1}+1)}; q^2)_{m_k-1}.
\]
Lastly, the peak is generated by $q^M_k$.  Together, this gives that 
\begin{multline*}\label{SCUgen_fn1}
SCU_k(q) = \\
\sum_{m_1,...,m_k\geq 1} q^{2(M_1 +  \cdots + M_{k-1}) + M_k}
(-q^2;q^2)_{m_1-1}(-q^{2(M_1+1)};q^2)_{m_2-1} \cdots (-q^{2(M_{k-1}+1)};q^2)_{m_k-1}.
\end{multline*}
We next observe that 
\[
(-q^2;q^2)_{m_1-1}(-q^{2(M_1+1)};q^2)_{m_2-1} \cdots (-q^{2(M_{k-1}+1)};q^2)_{m_k-1} = \frac{(q^2;q^2)_{M_k-1}}{(1+q^{2M_1}) \cdots (1+q^{2M_{k-1}})},
\]
which allows us to rewrite $SCU_k(q)$ as
\begin{multline}\label{SCUgen_fn2}
SCU_k(q) =
\sum_{m_1,...,m_k\geq 1} q^{2(M_1 +  \cdots + M_{k-1}) + M_k} \frac{(q^2;q^2)_{M_k-1}}{(1+q^{2M_1}) \cdots (1+q^{2M_{k-1}})} \\
= \sum_{M_k\geq k} q^{M_k}(-q^2;q^2)_{M_k-1}\sum_{1\leq M_1 < M_2 < \cdots < M_{k}} \dfrac{q^{2(M_1 + \cdots + M_{k-1})}}{(1+q^{2M_1})\cdots(1+q^{2M_{k-1}})}.
\end{multline}

We now interpret the right hand side of \eqref{SCUgen_fn2} combinatorially.  From the perspective of Remark \ref{mocktheta}, we see that for each choice of $M_k \geq k$, the term $q^{M_k}(-q^2;q^2)_{M_k-1}$ in \eqref{SCUgen_fn2} generates self-conjugate strongly unimodal sequences with peak of size $M_k$, or equivalently partitions into $M_k$ odd parts where each odd part at most the size of the largest part must occur at least once.  Then, in the inner sum, expand each factor as
\[
\frac{q^{2M_j}}{(1+q^{2M_j})} =  q^{2M_j}-q^{2(2M_j)}+q^{3(2M_j)}-q^{4(2M_j)} + \cdots,
\]
and interpret 
\[
\sum_{1\leq M_1 < M_2 < \cdots < M_{k}} \dfrac{q^{2(M_1 + \cdots + M_{k-1})}}{(1+q^{2M_1})\cdots(1+q^{2M_{k-1}})}
\]
as generating the difference in the number of partitions of some $n$ into an odd number of even parts $2M_1, \ldots, 2M_{k-1}$ minus the number of partitions of $n$ into an even number of even parts $2M_1, \ldots, 2M_{k-1}$, where each part $2M_j$ is $j$-marked, and $1\leq M_1 < M_2 < \cdots < M_{k}$.  Putting these interpretations together gives the result.
\end{proof}

\section{Concluding Remarks}\label{conclusion}

There are many potential directions to pursue in the study of $U_k(x_1, x_2,..., x_k; q)$.  For one, it is natural in the context of our discussions in \S \ref{intro-p} and \S \ref{intro-su}, to ask whether $U_k(x_1, x_2,..., x_k; q)$ possesses modularity properties of mock and/or quantum type when $(x_1, \ldots, x_k)$ is specialized at vectors of roots of unity.  As part of our REU project, we focused on potential quantum modularity properties.  We were able to determine a rational domain for $U_k(x_1, x_2,..., x_k; q)$ for certain vectors of roots of unity, and establish a trivial transformation property (see \cite[Section 4]{AKSS_Proc}).  However, a barrier for us to make more substantial headway is that we only have the multi-sum generating function for $U_k(x_1, x_2,..., x_k; q)$ given in Theorem \ref{UnDef} to work with.  This is in stark contrast to the situation for $k$-marked Durfee symbols, for which a $q$-hypergeometric transformation yields a beautiful single sum generating function for $R_k(x_1, x_2,..., x_k; q)$.  Despite our efforts, we were not able to produce a single sum generating function for $U_k(x_1, x_2,..., x_k; q)$.  It would be of great interest if a single sum generating function for $U_k(x_1, x_2,..., x_k; q)$ were discovered.

\end{document}